\documentclass[11pt,final]{amsart}
\usepackage[utf8]{inputenc}

\usepackage[OT2,T1]{fontenc}
\DeclareSymbolFont{cyrletters}{OT2}{wncyr}{m}{n}
\DeclareMathSymbol{\Sha}{\mathalpha}{cyrletters}{"58}

\usepackage{hyperref}
\usepackage[margin=1.3in]{geometry}
\usepackage{dcpic,pictex}
\usepackage{tikz-cd}  

\usepackage{amsfonts}
\usepackage{amssymb}
\usepackage{mathrsfs}
\usepackage{amsmath,amsthm}

\usepackage{enumerate}
\usepackage{enumitem}

\usepackage{color}
\usepackage{indentfirst}

\theoremstyle{plain} 
\newtheorem{thm}{Theorem}[section]
\newtheorem*{thm*}{Theorem}
\newtheorem{cor}[thm]{Corollary}
\newtheorem{lem}[thm]{Lemma}
\newtheorem{prop}[thm]{Proposition}
\numberwithin{equation}{section}

\theoremstyle{definition}
\newtheorem*{defn*}{Definition}
\newtheorem*{rem*}{Remark}
\newtheorem*{exa*}{Examples}

\newcommand{\Z}{\mathbb{Z}}
\newcommand{\Q}{\mathbb{Q}}
\newcommand{\Tr}{\mathrm{Tr}}
\newcommand{\Qtr}{\mathrm{Qtr}}
\newcommand{\Dec}{\mathrm{Dec}}
\newcommand{\Gal}{\mathrm{Gal}}

\title{The 3-rd unramified cohomology for norm one torus}

\author{Hanqing Long}

\address{Academy of Mathematics and System Science, CAS, Beijing 100190,
	P.\ R.\ China}
\email{hanqinglong@amss.ac.cn}

\author{Dasheng Wei}

\address{Hua Loo-Keng Key Laboratory of Mathematics,
	Academy of Mathematics and System Science, CAS, Beijing 100190,
	P.\ R.\ China \& School of mathematical Sciences, University of  CAS, Beijing
	100049, P.\ R.\ China}
\email{dshwei@amss.ac.cn}

\date{September 30, 2022}

\begin{document}
	
\maketitle

\bibliographystyle{alpha}

\begin{abstract}
For an algebraic torus $S$, Blinstein and Merkurjev  have given an estimate of $3$-rd unramified cohomology
$\bar{H}^3_{nr}(F(S),\mathbb Q/\mathbb Z(2))$ obtained from a flasque resolution of $S$. 
Based on their work, for the norm one torus $W=R_{K/F}^{(1)}\mathbb{G}_m$ with $K/F$ abelian, we compute the $3$-rd unramified cohomology $\bar{H}^3_{nr}(F(W),\mathbb Q/\mathbb Z(2))$. 
\end{abstract}

\section{Introduction}

Let $X$ be an integral algebraic variety over a field $F$, i.e. a separated scheme and of finite type over $F$.
The $i$-th  unramified cohomology group of $X$ denoted by $H_{nr}^{i}(F(X)/F,\mathbb{Q/Z}(j))$ (see section \S \ref{sec:pre} ) is defined by a subgroup of $H^{i}(F(X),\mathbb{Q/Z}(j))$,
which were first introduced by Colliot-Thélène and Ojanguren \cite{colliot1989varietes}.
We will write $\bar{H}_{nr}^{i}(F(X)/F,\mathbb{Q/Z}(j))$ as  $H_{nr}^{i}(F(X)/F,\mathbb{Q/Z}(j))/H^{i}(F,\mathbb{Q/Z}(j))$.

The unramified cohomology groups are homotopy invariants and can be used as an invariant to disprove the rationality of an algebraic variety. 
The non-triviality of the group $\bar{H}_{nr}^{i}(F(X)/F,\mathbb{Q/Z}(j))$ implies the stably irrational property of $X$. 
Peyre \cite{peyre1993unramified}, Colliot-Thélène and Pirutka \cite{colliot2016hypersurfaces}, Schreieder \cite{schreieder2019stably} have given many examples of varieties which are not stably rational but unirational.
The  unramified cohomology groups also play a role on rational points. In \cite{harari2015weak,harari2013local}, a certain
subgroup of the third unramified cohomology group can control the defect of weak approximation for tori over $p$-adic function fields.

There are no general methods to estimate the  unramified cohomology groups for general variety $X$. Let $W$ be a torus over $F$.
In 1995, Colliot-Th\'el\`ene \cite[p. 39]{ct1995} raised the problem: for $n$ prime to $\mathrm{char}(F)$ and $i>0$ , determine the group $\bar{H}_{nr}^{i}(F(W)/F,\mu_n^{\otimes(i-1)})$.
The first unramified cohomology group $\bar{H}_{nr}^{1}(F(W)/F,\Z/n\Z)$ is trivial.
The second unramified cohomology group $\bar{H}_{nr}^{2}(F(W)/F,\mathbb{Q/Z}(1))$ coincides with the unramified Brauer group $\mathrm{Br}_{nr}(W)/\mathrm{Br}(F)$ which is isomorphic to  the group $\Sha_\omega^2(\hat{W})$ and has been extensively studied (see \cite{colliot1987principal} and \cite{karpilovsky1987schur}).
In 2013, Blinstein and Merkurjev investigated  Colliot-Th\'el\`ene's problem when $i=3$ and obtained an exact sequence for the group $\bar{H}_{nr}^{3}(F(W)/F,\mathbb{Q/Z}(2))$, see \cite[Proposition 5.9]{m2013}.

Let $K/F$ be a finite abelian extension and $W$ the norm one torus $R_{K/F}^{(1)}\mathbb{G}_{m,K}$. Based on Blinstein and Merkurjev's work, we explicitly determine  the $p$-primary component of $\bar{H}_{nr}^{3}(F(W)/F,\mathbb Q/\mathbb Z(2))$  which we denote by $\bar{H}_{nr}^{3}(F(W)/F,\mathbb Q/\mathbb Z(2))\{p\}$ for any odd prime~$p$; furthermore, if $K/F$ has odd degree, we completely determine the group $\bar{H}_{nr}^{3}(F(W)/F,\mathbb Q/\mathbb Z(2))$. To our knowledge, our result is the first one of tori whose third unramified cohomology groups are completely determined and nonzero.

\begin{thm*}
	 Let $W=R_{K/F}^{(1)}(\mathbb G_{m,K})$ and $K/F$  an abelian extension  with $\mathrm{Gal}(K/F)=G\simeq\oplus_{i=1}^mC_i$ such that each $C_i$ is cyclic and $\#C_{i-1}|\#C_i$ $(1<i\leq m)$. Denote $d_i=(m-i)(m-i-1)/2$.  
	Then for any odd prime $p$, 
	$$\bar{H}^3_{nr}(F(W),\mathbb Q/\mathbb Z(2))\{p\}\simeq H^3(G,K^*)\{p\}  \oplus\oplus_{i=1}^{m-2}C_i\{p\}^{d_{i}}.
	$$
	Moreover, if $K/F$ has odd degree, then
	$$
	\bar{H}^3_{nr}(F(W),\mathbb Q/\mathbb Z(2))\simeq H^3(G,K^*) \oplus \oplus_{i=1}^{m-2}C_i^{d_{i}}.
	$$
\end{thm*}

If $G=\Z/n\oplus \Z/n$ and $n$ is odd, then $\bar{H}^3_{nr}(F(W),\mathbb Q/\mathbb Z(2)) \simeq H^3(G,K^*)$. 
In particular, if $F$ is a local field or some special number field (see Examples \ref{exa:un}), then $\bar{H}^3_{nr}(F(W),\mathbb Q/\mathbb Z(2))=0$. However, $\bar{H}^2_{nr}(F(W),\mathbb Q/\mathbb Z(1))=\Z/n$ is always  nontrivial.

\bigskip
{\bf Acknowledgements. }
The second author is supported by National Key R$\&$D Program of China.  The authors thank anonymous reviewers for their suggestions on Proposition \ref{prop1}.

\section{Preliminaries}\label{sec:pre}

In this section, we will introduce the unramified cohomology groups. 
The notation and definition coincide with \cite{m2013,ct1995,ct1977}.

Let $F$ be the base field (of arbitrary characteristic) and $\Gamma=\Gal(F^{sep}/F)$ the absolute Galois
group of $F$.
Given an integer $j>0$, 
we denote $\Q/\Z(j)$ as the direct sum of  
$$
\Q_p/\Z_p(j)=
\begin{cases}
	\underset{n}{\mathrm{colim}}\ \mu_{p^n}^{\otimes j}, \quad p\neq \mathrm{char}(F)\\
	\underset{n}{\mathrm{colim}}\ W_n\Omega_{\log}^j[-j], \quad p= \mathrm{char}(F)
\end{cases}
$$
in the derived category of sheaves of abelian groups on the big \'etale site of $\mathrm{Spec}\ F$ over all primes $p$, where $ W_n\Omega_{\log}^j[-j]$ is the sheaf of logarithmic de Rham–Witt differentials (see \cite{ASENS_1979_4_12_4_501_0} and \cite{m2013}).

Given a field extension $L/F$, let $A$ be a rank one discrete valuation ring with fraction field $L$
which contains $F$. 
The $i$-th unramified cohomology group of $L/F$ is defined by the group
$$
H_{nr}^i(L/F,\Q/\Z(j))):=\bigcap_{A\in P(L)}\mathrm{im}\ (H^i(A,\Q/\Z(j))\longrightarrow H^{i}(L,\Q/\Z(j)))
$$
where $P(L)$ is the set of all rank one discrete valuation rings which contains $F$ and has quotient field $L$.
They can also be defined by the intersection of the kernel of residue map $\partial_A$ for $A\in P(L)$ when $\mathrm{char}(F)=0$ \cite{ct1995}.

If $X$ is a smooth integral variety over $F$, the $i$-th unramifiel cohomology group of $X$ is defined by $H_{nr}^i(F(X)/F,\Q/\Z(j))$.
These groups are only dependent on the function field of $X$, and when $X$ is proper, we can replace $P(L)$ by the set of local rings of codimension one points in $X$ \cite{ct1995}. 
If $i=2$, the unramified chomology group coincides with the unramified Brauer group of $X$.
We will denote $\bar{H}_{nr}^{i}(F(X)/F,\mathbb{Q/Z}(j))$ by $H_{nr}^{i}(F(X)/F,\mathbb{Q/Z}(j))/H^{i}(F,\mathbb{Q/Z}(j))$.

Recall that an algebraic torus of dimension $n$ over $F$ is an algebraic group $T$ such that $T_{sep}$ is isomorphic to $\mathbb{G}_m^{n}$. 
For an algebraic torus $T$, we write $\hat{T}_{sep}$ as the $\Gamma$-module $\mathrm{Hom}(T_{sep},\mathbb{G}_m)$ and call it character lattice of $T$. 
The contravariant functor $T \mapsto \hat{T}_{sep}$ is an anti-equivalence between the
category of algebraic torus over $F$ and the category of $\Gamma$-lattices. 
We write $\hat{T}$ as the character group $\mathrm{Hom}_F(T,\mathbb G_m)$ and $T^\circ$ as the dual torus of $T$ such that $(\hat{T}^\circ)_{sep}=(\hat{T}_{sep})^\circ$.

Let $K/F$ be a finite Galois extension with $G=\mathrm{Gal}(K/F)$. 
We embed $K$ to $F_{sep}$, the separable closure of $F$, such that the absolute Galois group $\Gamma_K$ is a subgroup of $\Gamma$. We mainly consider the unramified cohomology of the norm one torus  $W=R_{K/F}^{(1)}(\mathbb G_{m,K})$. It is defined by the equation $N_{K/F}(\Xi)=1$ and can be constructed by the exact sequence
$$
1\longrightarrow W\longrightarrow R_{K/F}(\mathbb G_{m,K})\longrightarrow\mathbb G_{m}\longrightarrow 1
$$
and for its character lattice, we have
$$
0\longrightarrow\mathbb Z\stackrel{N_G}{\longrightarrow} \mathbb Z[G]\longrightarrow\hat{W}_{sep}\longrightarrow 0
$$
where $N_G$ maps  $1$ to $\sum_{g\in G}g$. 

A torus $T$ is called quasisplit if its character lattice $\hat{T}_{sep}$ is a permutation $\Gamma$-module, i.e. isomorphic to the finite direct-sum of $\mathbb Z[\Gamma/\Gamma_i]$ for open subgroups $\Gamma_i\leq \Gamma$.
And $T$ is called flasque (resp. coflasque) if $H^1(L,\hat{T}_{sep}^\circ)=0$ (resp. $H^1(L,\hat{T}_{sep})=0$) for all finite field extension $L/F$. 
A flasque resolution (resp. coflasque resolution) of a torus $T$ is an exact sequence of torus 
$$
1\longrightarrow S\longrightarrow P\longrightarrow T\longrightarrow 1
$$
(resp. $1\longrightarrow T\longrightarrow P\longrightarrow M\longrightarrow 1$ ) where $P$ is quasisplit and $S$ is flasque (resp. $M$ is coflasque).

For a torus $S$ over $F$, Blinstein and Merkurjev \cite{m2013} give an estimate of $\bar{H}^3_{nr}(F(S),\mathbb Q/\mathbb Z(2))$:

\begin{thm}(\cite{m2013}, Proposition 5.9)\label{mainlem}
Let $S$ be a torus over $F$ and let 
$$
1\longrightarrow T\longrightarrow P\longrightarrow S\longrightarrow 1
$$
be a flasque resolution of $S$. Then we have an exact sequence
\begin{align*}
0\longrightarrow CH^2(BT)_{tors}\longrightarrow H^1(F,T^\circ)\longrightarrow&\bar{H}^3_{nr}(F(S),\mathbb Q/\mathbb Z(2))\\
\longrightarrow& H^0(F,S^2(\hat{T}_{sep}))/\Dec\longrightarrow H^2(F,T^\circ).
\end{align*}
For an odd prime $p$, there is a canonical direct sum decomposition
$$
\bar{H}^3_{nr}(F(S),\mathbb Q/\mathbb Z(2))\{p\}\simeq H^1(F,T^\circ)\{p\}\oplus (H^0(F,S^2(\hat{T}_{sep}))/\Dec)\{p\}.
$$
\end{thm}

The first term is the torsion subgroup of the second Chow group of the classify variety of $T$. 
The symbol $\Dec$ in the fourth term is the subgroup $\Dec(H^0(F,S^2(\hat{T}_{sep})))$ of $H^0(F,S^2(\hat{T}_{sep}))$, defined as follows. 
Given a $\Gamma$-module $A$, $\Dec(A)\subset S^2(A)$ is generated by $(A^{\Gamma})^2$ and $\Qtr_{\Gamma'}(A^{\Gamma'})$ for all open subgroups $\Gamma'\subset\Gamma$, with the following definition
\begin{align*}
\Qtr_{\Gamma'}:\ A^{\Gamma'}\to&\ S^2(A)^{\Gamma}\\
a\mapsto&\ \sum_{i<j}\sigma_i(a)\cdot \sigma_j(a)
\end{align*}
where $\{\sigma_i\}$ is a representative for the left cosets of $\Gamma'$ in $\Gamma$.

\begin{prop}\label{prop1}
	Let $F$ be a number field, then $\bar{H}^3_{nr}(F(S),\mathbb Q/\mathbb Z(2))$ is a finite group for any algebraic torus $S$ over $F$.
\end{prop}

\begin{proof}
	 The group $\bar{H}^3_{nr}(F(S),\mathbb Q/\mathbb Z(2))$ is killed by the degree of the splitting field of $S$ over $F$, and $\bar{H}^3_{nr}(F(S),\mathbb Q/\mathbb Z(2))/H^1(F,T^\circ)$ is finitely generated by Theorem \ref{mainlem}.
	Therefore, it suffices to prove that $H^1(F,T)$ is finite for any coflasque torus $T$, here in fact $T$ is the dual of the flasque torus $T$ in Theorem \ref{mainlem}. 
	Indeed, its proof is a variant of the proof of   \cite[Theorem 1]{ct1977} which shows the finiteness of  $H^1(F,T)$  for any flasque torus. Let $K/F$ be a finite Galois extension which splits $T$, and $\Omega_{F}$ the set of places of $F$.
	Let $P\subset \Omega_{F}$ be a finite set which contains all infinite places and all ramified places in $K/F$ and $O_P=\{a\in K| \mathrm{ord}_w(a)\ge 0 \text{ for any }v\notin P \text{ and } w|v\}$ be the ring of $P$-integers. 
	We can enlarge $P$ such that the class number of $O_P$ is $1$.
	This induces an exact sequence of $G$-modules $0\to O_P^*\to K^*\to \mathrm{Div}\ O_P\to 0$.
	By construction, $$\mathrm{Div}\ O_P\simeq \oplus_{v\notin P}\oplus_{w|v}\mathbb{ Z}w\simeq\oplus_{v\notin P}\mathbb{ Z}[G/G_v]$$
	where $G_v=\Gal(K_w/F_v)$ is cyclic since $v$ is unramified. 
	Tensoring by $\hat{T}^\circ$ over $\mathbb{ Z}$, it is still exact and hence the following is exact:
	$$
	H^1(G,\hat{T}^\circ\otimes_{\mathbb Z}O_P^*)\to H^1(G,\hat{T}^\circ\otimes_{\mathbb Z}K^*)\to H^1(G,\hat{T}^\circ\otimes_{\mathbb Z}\oplus_{v\notin P}\mathbb{ Z}[G/G_v]).
	$$
	The first term is finite since $\hat{T}^\circ\otimes_{\mathbb Z}O_P^*$ is finitely generated. The second is no other than $H^1(G,T)\simeq H^1(F,T)$. The third is $0$ because $G_v$ is cyclic and $\hat{T}^\circ$ is flasque: $H^1(G,\hat{T}^\circ\otimes_{\mathbb Z}\mathbb{ Z}[G/G_v])\simeq H^1(G_v,\hat{T}^\circ)\simeq H^{-1}(G_v,\hat{T}^\circ)=0$.
\end{proof}

\section{Main results}

We continue to use the notations in \S \ref{sec:pre}. 
For the norm one torus $W=R^{(1)}_{K/F}(\mathbb{G}_{m,K})$ of Galois extension $K/F$,
we fix a flasque resolution of $W$ constructed from \cite[Proposition 15]{ct1977}:
\begin{equation}\label{seq:flasque}
1\longrightarrow T\longrightarrow P\longrightarrow W\longrightarrow 1 
\end{equation}
where $P=R_{K/k}(\mathbb{G}_{m,K})^r$ and $r$ is the number of generators of $G=\Gal(K/F)$.
We mainly consider the $3$-th unramified cohomology group
$\bar{H}^3_{nr}(F(W),\mathbb Q/\mathbb Z(2))$ and it can be expressed by $T$ from Theorem \ref{mainlem}. 
Building on this work, we will reduce $\bar{H}^3_{nr}(F(W),\mathbb Q/\mathbb Z(2))$ to a form determined only by the field extension $K/F$ and the Galois group $\mathrm{Gal}(K/F)$.

\begin{lem}\label{firstterm}
Let $W=R_{K/F}^{(1)}(\mathbb G_{m,K})$ be a norm one torus of  the finite Galois extension $K/F$ with $G=\mathrm{Gal}(K/k)$ and let 
$T$ be the torus in the flasque resolution (\ref{seq:flasque}) of $W$. Then we have $H^1(F,T^\circ)\simeq  H^3(G,K^*)$.
\end{lem}
\begin{proof}
Firstly, we have an isomorphism $H^1(F,T^\circ)\simeq H^1(G,T^\circ(K))$. 
The following exact sequence is derived from (\ref{seq:flasque}):
$$
1\longrightarrow W^\circ(K)\longrightarrow P^\circ(K)\longrightarrow T^\circ(K)\longrightarrow 1.
$$
One has $H^i(G,P^\circ(K))=0$ for $i>0$ since $R_{K/F}(\mathbb G_{m,K})(K)\simeq \mathbb Z[G]\otimes_{\mathbb Z}K^*$ is an induced module. 
Hence we obtain an isomorphism $ H^1(G,T^\circ(K))\simeq H^2(G,W^\circ(K))$. 
On the other hand, the following sequence is exact 
$$
1\longrightarrow\mathbb  G_{m}(K)\longrightarrow R_{K/F}(\mathbb G_{m,K})(K)\longrightarrow W^\circ(K)\longrightarrow 1. 
$$
The proof then follows from  $$H^1(F,T^\circ)\simeq H^1(G,T^\circ(K))\simeq H^2(G,W^\circ(K))\simeq H^3(G,K^*).\qedhere$$
\end{proof}

From the flasque resolution (\ref{seq:flasque}) of $W=R^{(1)}_{K/F}(\mathbb G_{m,K})$, we can obtain the following exact sequence of dual $\Gamma$-modules
\begin{equation}
1\longrightarrow \hat{W}_{sep}\longrightarrow \hat{P}_{sep}\longrightarrow \hat{T}_{sep}\longrightarrow 1.\label{exactseq}
\end{equation}
The part $H^0(F,S^2(\hat{T}_{sep}))/\Dec$ in Theorem \ref{mainlem} is actually only dependent on the group $\mathrm{Gal}(K/F)$. 
To prove this, we need several lemmas. 

\begin{lem}
Let $N$ be the kernel of the natural surjective homomorphism $S^2(\hat{P}_{sep})\to S^2(\hat{T}_{sep})$, then the sequence 
$$
0\rightarrow\mathbb\wedge^2\hat{W}_{sep}  \stackrel{f}{\rightarrow} \hat{W}_{sep}\otimes\hat{P}_{sep}\stackrel{f'}{\rightarrow}  N\rightarrow 0
$$
is exact where $f(a\wedge b)=a\otimes b-b\otimes a$ and $f'$ is induced by $\hat{W}_{sep}\otimes\hat{P}_{sep}\longrightarrow S^2(\hat{P}_{sep})$.
\end{lem}

\begin{proof}
The exactness at $\wedge^2\hat{W}_{sep}$ and $\hat{W}_{sep}\otimes\hat{P}_{sep}$ is obvious. We only need to show that $f'$ is surjective.  
Assume that $\hat{T}_{sep}=<\beta_1,\cdots,\beta_t>$, and $\beta_i$ is the image of $\alpha_i\in\hat{P}_{sep}$, so $\alpha_1,\cdots,\alpha_t$ is $\mathbb Z$-linear independent in $\hat{P}_{sep}$. 
Consider the following commutative diagram:
\[
\begin{tikzcd}
\hat{P}_{sep}\otimes\hat{P}_{sep}\arrow[r,"g"]\arrow[d,"h"]&\hat{T}_{sep}\otimes\hat{T}_{sep}\arrow[d,"h'"]\\
S^2(\hat{P}_{sep})\arrow[r,"g'"]&S^2(\hat{T}_{sep})
\end{tikzcd}
\]
where $h$, $h'$ are the natural homomorphisms and $g$, $g'$ are induced by the exact sequence (\ref{exactseq}).
It is not difficult to see that $g$, $g'$, $h$, $h'$ are surjective.
If $\alpha\in \hat{P}_{sep}\otimes\hat{P}_{sep}$ such that $g'\circ h(\alpha)=0\in S^2(\hat{T}_{sep}),$ this implies $$ g(\alpha)=\sum_{i<j}b_{ij}(\beta_i\otimes\beta_j-\beta_j\otimes\beta_i)\in \hat{T}_{sep}\otimes\hat{T}_{sep}, b_{ij}\in \mathbb Z.$$
Hence $$\alpha=\sum_{i<j}b_{ij}((\alpha_i+\omega_i)\otimes(\alpha_j+\omega_j)-(\alpha_j+\omega_j')\otimes(\alpha_i+\omega_i')), \ b_{ij}\in \mathbb Z,$$ where $\omega_i, \omega_i', \omega_j, \omega_j'\in \hat{W}_{sep}$. Then we can deduce $$\alpha=\alpha'+\alpha'',\alpha'\in \{a\otimes b-b\otimes a|a,b\in \hat{P}_{sep}\},\alpha''\in \hat{W}_{sep}\otimes\hat{P}_{sep}+\hat{P}_{sep}\otimes\hat{W}_{sep}.$$ Therefore, the image of $\alpha$ in $S^2(\hat{P}_{sep})$ comes from $\hat{W}_{sep}\otimes\hat{P}_{sep}$, thus $f'$ is surjective.
\end{proof}

Because $\Gamma_K$ acts trivially on $S^2(\hat{P}_{sep})$, we can only consider the  $G$-action:  $S^2(\hat{T}_{sep})^{\Gamma} = S^2(\hat{T}_{sep})^G$.
The short exact sequence
$
0\longrightarrow N\longrightarrow S^2(\hat{P}_{sep}) \stackrel{g'}{\longrightarrow} S^2(\hat{T}_{sep})\longrightarrow 0
$ induce an exact sequence:
\begin{equation}
S^2(\hat{P}_{sep})^G \stackrel{g'}{\longrightarrow} S^2(\hat{T}_{sep})^G\stackrel{\delta}{\longrightarrow} H^1(G,N)\longrightarrow H^1(G,S^2(\hat{P}_{sep})).\label{dec}
\end{equation}

\begin{lem}\label{duichen}
If $G$ is a finite group and $p$ is an odd prime number, then $$H^i(G,S^2(\mathbb Z[G]))\{p\}=H^i(G,\wedge^2\mathbb Z[G])\{p\}= 0$$ for $i\ge1$. Furthermore, if $G$ has odd order, then $$H^i(G,S^2(\mathbb Z[G]))=H^i(G,\wedge^2\mathbb Z[G])= 0$$ for $i\ge1$.
\end{lem}

\begin{proof}
We have a commutative diagram with the row exact:
\[
\begin{tikzcd}
&&S^2(\mathbb Z[G])\arrow[r,equal]\arrow[d,"i_1"]&S^2(\mathbb Z[G])\arrow[d,"\times 2"]&\\
0\arrow[r]&\wedge^2\mathbb Z[G]\arrow[r,"i_2"]\arrow[d,"\times 2"]&\mathbb Z[G]^{\otimes 2}\arrow[r,"p_1"]\arrow[d,"p_2"]&S^2(\mathbb Z[G])\arrow[r]&0\\
&\wedge^2\mathbb Z[G]\arrow[r,equal]&\wedge^2\mathbb Z[G]&&
\end{tikzcd}
\]
where $i_1(a\cdot b)=a\otimes b+b\otimes a$, $i_2(a\wedge b)=a\otimes b-b\otimes a$, $p_1(a\otimes b)=a\cdot b$ and $p_2(a\otimes b)=a\wedge b$.

Therefore, this induces a commutative diagram with the row exact for $i\ge1$:
\[
\begin{tikzcd}
&H^i(G,S^2(\mathbb Z[G]))\arrow[r,equal]\arrow[d,"i_1^*"]&H^i(G,S^2(\mathbb Z[G]))\arrow[d,"\times2"]\\
H^i(G,\wedge^2\mathbb Z[G])\arrow[r,"i_2^*"]\arrow[d,"\times2"]&H^i(G,\mathbb Z[G]^{\otimes 2})\arrow[r,"p_1^*"]\arrow[d,"p_2^*"]&H^i(G,S^2(\mathbb Z[G]))\\
H^i(G,\wedge^2\mathbb Z[G])\arrow[r,equal]&H^i(G,\wedge^2\mathbb Z[G])&
\end{tikzcd}
\]
Since $H^i(G,\mathbb Z[G]^{\otimes 2})=0$, one obtains $2\cdot  H^i(G,S^2(\mathbb Z[G])) =2\cdot  H^i(G,\wedge^2\mathbb Z[G])= 0.$
This implies $H^i(G,S^2(\mathbb Z[G]))\{p\}=H^i(G,\wedge^2\mathbb Z[G])\{p\}= 0$ for $i\ge 1$ and  odd prime $p$. If $G$ has odd order, then $H^i(G,S^2(\mathbb Z[G]))\{p\}$ and $H^i(G,\wedge^2\mathbb Z[G])\{p\}$ have odd order, the proof then follows.
\end{proof}

\begin{lem}\label{invertibletorus}
	Let $P$ be the quasisplit torus, then $S^2(\hat{P}_{sep})^G/\Dec=0$.
\end{lem}
\begin{proof}  
	By the definition of $\Dec$, we can deduce (see \cite[A-II]{m2013}):
	\begin{align*}S^2(A\otimes B)^\Gamma&\simeq  S^2(A)^\Gamma\oplus (A\otimes B)^\Gamma\oplus S^2(B)^\Gamma,\\
	\Dec(A\oplus B)&\simeq \Dec(A)\oplus \Dec(A,B)\oplus \Dec(B),
	\end{align*}
	where $\Dec(A,B)$ is the subgroup of $(A\otimes B)^{\Gamma}$ generated by elements $\Tr(a\otimes b)=\sum_{i}\sigma_i a\otimes\sigma_i b$ for all open subgroups $\Gamma'$ of $\Gamma$ and all $a\in A^{\Gamma'}$ and $b\in B^{\Gamma'}$, where $\{\sigma_i\}$ is a representative of $\Gamma/\Gamma'$. Assume that $A$ and $B$ are permutation modules, it is clear that $(A\otimes B)^{\Gamma}=\Dec(A, B)$. 
	So we can reduce to the case $\hat{P}_{sep}=\mathbb Z[\Gamma/\Gamma_L]$ by induction, where $L/F$ is a finite field extension and $\Gamma_L$ is the absolute Galois group of $L$.
	
	We denote  $\bar \sigma_1\cdot \bar \sigma_2$ to be the image of $\bar \sigma_1\otimes \bar \sigma_2$ in $S^2(\mathbb Z[\Gamma/\Gamma_L])$, where $\bar \sigma_i \in \Gamma/\Gamma_L$ and $i=1,2$. Let $<\bar \sigma_1\cdot \bar \sigma_2>$ be the orbit of $\Gamma$-action on $\bar \sigma_1\cdot \bar \sigma_2$. Let $N_{\bar \sigma}:=\sum_{a\in <\bar e\cdot \bar \sigma>}a$. It is clear that $S^2(\mathbb Z[\Gamma/\Gamma_L])^{\Gamma}$ is generated by all $N_{\bar \sigma},\ \bar \sigma\in \Gamma/\Gamma_L$. It suffices to show $N_{\bar \sigma}\in \Dec(\mathbb Z[\Gamma/\Gamma_L])$.
	
	Let $\Gamma'=\Gamma_L\cap \Gamma_{\sigma(L)}$ and $\Gamma_{\bar \sigma}$ the stabilizer of $\bar e\cdot \bar \sigma  $, which is an open subgroup of $\Gamma$ containing $\Gamma'$. 
\begin{lem}\label{lem:stabilizer} 
	Either  $\Gamma_{\bar \sigma}=\Gamma'$ or $\Gamma_{\bar \sigma}=\Gamma'\cup \sigma'\Gamma' $, where $\sigma'\Gamma_L= \sigma\Gamma_L$ and $\sigma'\not\in \Gamma'$.
\end{lem}	
\begin{proof} We only need to prove the case $\Gamma_{\bar \sigma}\neq\Gamma'$.
 Assume that $\sigma',\sigma''\in \Gamma_{\bar \sigma}\setminus\Gamma'$. So ${\sigma'\bar\sigma}={\sigma''\bar\sigma}=\bar e$	 and $\bar \sigma'=\bar \sigma''=\bar \sigma$. 
 Then $\sigma'^{-1}\sigma''\in \Gamma_{\sigma(L)}\cap \Gamma_L=\Gamma'$, hence $\sigma'\Gamma'=\sigma''\Gamma'$.\qedhere	
\end{proof}

Let $\bar \sigma\in \Gamma/ \Gamma_L$ and $\sigma_1,\cdots,\sigma_m$ representatives of $\Gamma/\Gamma'$.
We have \begin{align*}
	\Qtr_{\Gamma'}(\bar e+\bar \sigma)&=\sum_{i<j}\sigma_i(\bar e+\bar \sigma)\cdot\sigma_i(\bar e+\bar \sigma)\\
	&=\sum_{i<j}\bar \sigma_i\cdot \bar \sigma_j+\sum_{i<j}{\sigma_i\bar \sigma}\cdot {\sigma_j\bar \sigma}+\sum_{i<j}\bar \sigma_i\cdot  {\sigma_j\bar\sigma}+\sum_{i>j}\bar \sigma_i\cdot  {\sigma_j\bar\sigma}\\
&=\Qtr_{\Gamma'}(\bar e)+\Qtr_{\Gamma'}(\bar \sigma)+\mathrm {Tr}(\bar e)\cdot \mathrm {Tr}(\bar e)-\sum_{i}\bar\sigma_i\cdot\sigma_i\bar\sigma.
\end{align*}
Therefore, $\sum_{i}\bar\sigma_i\cdot\sigma_i\bar\sigma\in \Dec(\Bbb Z[\Gamma/\Gamma_L])$. If $\Gamma_{\bar \sigma}=\Gamma'$, then 
$N_{\bar \sigma}=\sum_{i}\bar\sigma_i\cdot\sigma_i\bar\sigma\in \Dec(\mathbb Z[\Gamma/\Gamma_L])$.

Suppose $[\Gamma_{\bar \sigma}:\Gamma']=2$. Let $\sigma'\in \Gamma_{\bar \sigma}\setminus \Gamma'$ and $\sigma_1,\cdots,\sigma_m$  representatives of $\Gamma/\Gamma_{\bar \sigma}$, hence $\sigma_1,\cdots,\sigma_m,\sigma_1\sigma',\cdots,\sigma_m\sigma'$ are  representatives of $\Gamma/\Gamma_{\bar \sigma}$. Since $\bar \sigma'=\bar\sigma\in \Gamma/\Gamma_L$ by Lemma~\ref{lem:stabilizer}, one has $\bar e+\bar \sigma'=\bar e+\bar \sigma\in \mathbb Z[\Gamma/\Gamma_L]^{\Gamma_{\bar\sigma}}$, hence
\begin{align*}
	\Qtr_{\Gamma_{\bar \sigma}}(\bar e+\bar \sigma')&=\sum_{i<j}\sigma_i(\bar e+\bar \sigma')\cdot \sigma_j(\bar e+\bar \sigma')\\
	&=\sum_{i<j}\bar \sigma_i\cdot \bar \sigma_j+\sum_{i<j}{\sigma_i\bar \sigma'}\cdot {\sigma_j\bar \sigma'}+\sum_{i<j}\bar \sigma_i\cdot  {\sigma_j\bar\sigma'}+\sum_{i>j}\bar \sigma_i\cdot  {\sigma_j\bar\sigma'}\\
&=\Qtr_{\Gamma'}(\bar e)-\sum_{i}\bar\sigma_i\cdot{\sigma_i\bar \sigma},
\end{align*}
which implies $N_{\bar \sigma}=\sum_{i}\bar\sigma_i\cdot{\sigma_i\bar \sigma}\in \Dec(\mathbb Z[\Gamma/\Gamma_L])$, the proof of Lemma \ref{lem:stabilizer} then follows.
\end{proof}

\begin{rem*}
	Let  $S=\mathrm{Spec}\ F$ be the trivial torus.
	Then $
	1\longrightarrow P\longrightarrow P\longrightarrow S\longrightarrow 1$ 
	is a flasque resolution of $S$. Let $p$ be an odd prime. Therefore
	$$
	0=\bar{H}^3_{nr}(F(S),\mathbb Q/\mathbb Z(2))\{p\}\simeq H^1(F,P^\circ)\{p\}\oplus (H^0(F,S^2(\hat{P}_{sep}))/\Dec)\{p\}
	$$
	by Theorem \ref{mainlem}.
	Hence $(S^2(\hat{P}_{sep})^G/\Dec)\{p\}=0$.
	However, the $2$-primary part cannot be determined by this method.
\end{rem*}

\begin{lem}\label{commutator}
If $\#[G,G]$ is a power of $2$, then for an odd prime number $p$
$$(H^0(F,S^2(\hat{T}_{sep}))/\Dec)\{p\}\simeq  (H^1(G,N)/\delta((\hat{T}_{sep}^G)^{\otimes2}))\{p\},$$ where $N,\delta$ are defined in (\ref{dec}).
\end{lem}

\begin{proof}
Let $H$ be a subgroup of $G$, $G^{ab}=G/[G,G]$ and $H^{ab}=H/[H,H]$. 
Since $\#[G,G]$ is a power of $2$, the exact sequence $
0\longrightarrow\mathbb Z\stackrel{N}{\longrightarrow} \mathbb Z[G]\longrightarrow\hat{W}_{sep}\longrightarrow 0
$ deduces the following commutative diagram
\[
\begin{tikzcd}
H^1(G,\hat{W}_{sep})\arrow[r,"\simeq"]\arrow[d,"j"]&H^2(G,\mathbb Z)\arrow[r,"\simeq"]\arrow[d]&H^1(G,\mathbb Q/\mathbb Z)\arrow[d]\arrow[r,"\simeq"]&H^1(G^{ab},\mathbb Q/\mathbb Z)\arrow[d,"j'"]\\
H^1(H,\hat{W}_{sep})\arrow[r,"\simeq"]&H^2(H,\mathbb Z)\arrow[r,"\simeq"]&H^1(H,\mathbb Q/\mathbb Z)\arrow[r,"\simeq"]&H^1(H^{ab},\mathbb Q/\mathbb Z)
\end{tikzcd}
\]
Therefore, the cokernel $\mathrm{Coker}(j)\simeq \mathrm{Coker}(j')$ is a $2$-group.

The sequence (\ref{exactseq}) derives the following commutative diagram with the row  exact
\[
\begin{tikzcd}
	0\arrow[r]&\hat{W}_{sep}^G\arrow[r]\arrow[d]&\hat{P}_{sep}^G\arrow[r]\arrow[d,"r"]&\hat{T}_{sep}^G\arrow[d,"r'"]\arrow[r]&H^1(G,\hat{W}_{sep})\arrow[r]\arrow[d,"j"]&H^1(G,\hat{P}_{sep})=0\\
	0\arrow[r]&\hat{W}_{sep}^H\arrow[r]&\hat{P}_{sep}^H\arrow[r]&\hat{T}_{sep}^H\arrow[r]&H^1(H,\hat{W}_{sep})\arrow[r]&H^1(H,\hat{P}_{sep})=0.
\end{tikzcd}
\]
Since the $\mathrm{Coker}(j)$ is a $2$-group and the sequence $ \mathrm{Coker}(r)\to  \mathrm{Coker}(r') \to\mathrm{Coker} (j)\to 0$ is  exact,
it shows that
 $2^m\hat{T}_{sep}^H\subset \mathrm{im}(\hat{P}_{sep}^H)+\hat{T}_{sep}^G\subset\hat{T}_{sep}^H$ for some integer $m$.

Let $\{\gamma_i\}$ be a  representative for the left cosets of $H$ in $G$. 
If $a\in \hat{T}_{sep}^H$, then $2^ma=b+c\in \hat{T}_{sep}^H, b\in  \mathrm{im}(\hat{P}_{sep}^H),c\in\hat{T}_{sep}^G$. By the commutative diagram
\[
\begin{tikzcd}
	\hat{P}_{sep}^H\arrow[r]\arrow[d,"\Qtr"]&\hat{T}_{sep}^H\arrow[d,"\Qtr"]\\
	S^2(\hat{P}_{sep})^G\arrow[r,"g"]&S^2(\hat{T}_{sep})^G,
\end{tikzcd}
\]
one obtains 
\begin{align*}
	\Qtr(2^m a)&=\sum_{i<j}\gamma_i(b+c)\cdot\gamma_j(b+c)\\
&=\sum_{i<j}\gamma_i b\cdot\gamma_j b+\sum_{i<j}c\cdot c+lc\cdot\sum_{i}\gamma_ib,
\end{align*}
where $l=[G:H]-1$.
So  $$\Qtr(2^m\hat{T}_{sep}^H)=2^{2m}\Qtr(\hat{T}_{sep}^H)\subset g(\Dec(\hat{P}_{sep}))+(\hat{T}_{sep}^G)^{2}.$$
Let $H$ run through all subgroups of $G$, since $G$ is finite, there exists a  large enough integer $m'$ such that $2^{m'}\Dec(\hat{T}_{sep})\subset g(\Dec(\hat{P}_{sep}))+(\hat{T}_{sep}^G)^{2}$, hence
\begin{align*}
	(S^2(\hat{T}_{sep})^G/2^{2m'}\Dec(\hat{T}_{sep}))\{p\}&=  [S^2(\hat{T}_{sep})^G/(g'(\Dec(\hat{P}_{sep}))+(\hat{T}_{sep}^G)^{2})]\{p\}\\
	&= (S^2(\hat{T}_{sep})^G/\Dec(\hat{T}_{sep}))\{p\}.
\end{align*}

On the other hand,  by Lemma \ref{invertibletorus} and Lemma \ref{duichen}, 
the sequence \eqref{dec} derives the isomorphism
$(S^2(\hat{T}_{sep})^G/g(\Dec(\hat{P}_{sep})))\{p\}\stackrel{\delta}{\longrightarrow} H^1(G,N)\{p\}
$
thus $$(H^0(F,S^2(\hat{T}_{sep}))/\Dec(\hat{T}_{sep}))\{p\}\simeq  (H^1(G,N)/\delta((\hat{T}_{sep}^G)^{2}))\{p\}.\qedhere$$
\end{proof}

\begin{rem*}
	Assume that $G$ is abelian.
	Then $\Dec(\hat{T}_{sep})= g'(\Dec(\hat{P}_{sep}))+(\hat{T}_{sep}^G)^{2}$
	from the proof above. By Lemma \ref{invertibletorus} and the sequence (\ref{dec}), we deduce the following isomorphism 
	\begin{equation}
		S^2(\hat{T}_{sep})^G/\Dec\stackrel{\simeq}{\longrightarrow} H^1(G,N)/\delta((\hat{T}_{sep}^G)^{\otimes2}).\label{abelian}
	\end{equation}
\end{rem*}

Let $N'$ be the kernel of  $\hat{P}_{sep}\otimes\hat{T}_{sep}\to S^2(\hat{T}_{sep})$ and $N''$ the kernel of  $\hat{P}_{sep}\otimes\hat{P}_{sep}\to S^2(\hat{T}_{sep})$. 
Then we have the following commutative diagram with the exact rows:
\[
\begin{tikzcd}
0\arrow[r]&N\arrow[r]& S^2(\hat{P}_{sep})\arrow[r]& S^2(\hat{T}_{sep})\arrow[r]&0\\
0\arrow[r]&N''\arrow[r]\arrow[u]\arrow[d]& \hat{P}_{sep}\otimes\hat{P}_{sep}\arrow[r]\arrow[u]\arrow[d]& S^2(\hat{T}_{sep})\arrow[r]\arrow[u,equal]\arrow[d,equal]&0\\
0\arrow[r]&N'\arrow[r]& \hat{P}_{sep}\otimes\hat{T}_{sep}\arrow[r]& S^2(\hat{T}_{sep})\arrow[r]&0
\end{tikzcd}
\]

By construction $\hat{P}_{sep}\simeq \mathbb{Z}[G]^{\oplus r}$, for $i\ge 1$ and an odd prime $p$, $\hat{H}^i(G,\hat{P}_{sep}\otimes\hat{P}_{sep})=\hat{H}^i(G, \hat{P}_{sep}\otimes\hat{T}_{sep})=0$ by Shapiro's Lemma, $\hat{H}^i(G,S^2(\hat{P}_{sep}))\{p\}=0$ by Lemma \ref{duichen}. 
Therefore the following diagram is commutative
$$
\begin{tikzcd}
	&\hat{H}^1(G,N)\{p\}\\
	\hat{H}^0(G,S^2(\hat{T}_{sep}))\{p\}\arrow[dr, "\delta'", "\simeq"']\arrow[r, "\delta''", "\simeq"']\arrow[ur,"\delta", "\simeq"']&\hat{H}^1(G,N'')\{p\}\arrow[u, "\simeq"]\arrow[d, "\simeq"]\\
	&\hat{H}^1(G,N')\{p\}
\end{tikzcd}
$$
and we deduce 
\begin{align*}
(H^1(G,N)/\delta((\hat{T}_{sep}^G)^{\otimes2}))\{p\}&\simeq(H^1(G,N')/\delta'((\hat{T}_{sep}^G)^{\otimes2}))\{p\}\\
&\simeq(H^1(G,N'')/\delta''((\hat{T}_{sep}^G)^{\otimes2}))\{p\}.
\end{align*}

\begin{prop}\label{prop:cup}
Let $W=R_{K/F}^{(1)}(\mathbb G_{m,K})$ and  $K/F$ the finite Galois extension  with $G=\mathrm{Gal}(K/F)$. Let 
$$
1\longrightarrow T\longrightarrow P\longrightarrow W\longrightarrow 1
$$
be the flasque resolution (\ref{seq:flasque}) of $W$. Suppose the commutator $[G,G]$ is a $2$-group. Then
$$
(H^0(F,S^2(\hat{T}_{sep}))/\Dec)\{p\}\simeq  \mathrm{Coker}(H^2(G,\mathbb Z)\times H^2(G,\mathbb Z)\stackrel{\cup}{\longrightarrow}H^4(G,\mathbb Z))\{p\}.
$$
\end{prop}
\begin{proof}

By \cite[Chapter I, \S 4]{neukirch2013}, we have the following commutative diagram:
$$
\begin{tikzcd}
\hat{H}^0(G,\hat{T}_{sep})\arrow[d,"\simeq"]\arrow[phantom, r, description, "\times"]&\hat{H}^0(G,\hat{T}_{sep})\arrow[d,equal]\arrow[r,"\cup"]&\hat{H}^{0}(G,S^2(\hat{T}_{sep}))\arrow[d, "\delta'"]\arrow[dr, "\delta''"]\arrow[drr, "\delta"]\\
\hat{H}^{1}(G,\hat{W}_{sep})\arrow[d,equal]\arrow[phantom, r, description, "\times"]&\hat{H}^0(G,\hat{T}_{sep})\arrow[d,"\simeq"]\arrow[r,"\cup"]&\hat{H}^{1}(G,N')&\hat{H}^{1}(G,N'')\arrow[l,"\simeq"]\arrow[r,"\simeq"]&\hat{H}^{1}(G,N)\arrow[d,"\simeq"]\\
\hat{H}^1(G,\hat{W}_{sep})\arrow[d,"\simeq"]\arrow[phantom, r, description, "\times"]&\hat{H}^1(G,\hat{W}_{sep})\arrow[d,equal]\arrow[rrr,"\cup"]&&&\hat{H}^{2}(G,\wedge^2\hat{W}_{sep})\arrow[d,""]\\
\hat{H}^2(G,\mathbb Z)\arrow[d,equal]\arrow[phantom, r, description, "\times"]&\hat{H}^1(G,\hat{W}_{sep})\arrow[d,"\simeq"]\arrow[rrr,"\cup"]&&&\hat{H}^{3}(G,\hat{W}_{sep})\arrow[d,"\simeq"]\\
\hat{H}^2(G,\mathbb Z)\arrow[phantom, r, description, "\times"]&\hat{H}^2(G,\mathbb Z)\arrow[rrr,"\cup"]&&&\hat{H}^{4}(G,\mathbb Z)
\end{tikzcd}
$$

a. The first square is induced by the exact sequences:
$
1\longrightarrow \hat{W}_{sep}\longrightarrow \hat{P}_{sep}\longrightarrow \hat{T}_{sep}\longrightarrow 1
$
and
$
0\longrightarrow N'\longrightarrow\hat{P}_{sep}\otimes\hat{T}_{sep}\longrightarrow S^2(\hat{T}_{sep})\longrightarrow0
$. 
Note that $\hat{H}^0(G,\hat{P}_{sep})=\hat{H}^1(G,\hat{P}_{sep})=0$ 
thus the vertical arrows at left is an isomorphism.

b. The second square is induced by the exact sequences:
$
1\longrightarrow \hat{W}_{sep}\longrightarrow \hat{P}_{sep}\longrightarrow \hat{T}_{sep}\longrightarrow 1
$
and
$
0\longrightarrow\mathbb\wedge^2\hat{W}_{sep}  \stackrel{f}{\longrightarrow} \hat{W}_{sep}\otimes\hat{P}_{sep}\longrightarrow  N\longrightarrow 0
$. 
Note that $\hat{H}^1(G,\hat{W}_{sep}\otimes\hat{P}_{sep})=\hat{H}^2(G,\hat{W}_{sep}\otimes\hat{P}_{sep})=0$.

c. The third square is induced by the exact sequences:
$
0\longrightarrow\mathbb Z\stackrel{N_G}{\longrightarrow} \mathbb Z[G]\longrightarrow\hat{W}_{sep}\longrightarrow 0
$
and
$
0\longrightarrow\hat{W}_{sep}  \stackrel{\varphi}{\longrightarrow}\wedge^2\mathbb Z[G] \longrightarrow \wedge^2\hat{W}_{sep} \longrightarrow 0
$, where $\wedge^2\mathbb Z[G] \longrightarrow \wedge^2\hat{W}_{sep}$ are the natural surjective homomorphism.
We now define the hommomorphism $\varphi$ and prove the exactness at $\wedge^2\mathbb Z[G]$.
We write $\bar{g}\in\hat{W}_{sep}$ as the image of $g\in \Z[G]$ and $\bar{g}\wedge\bar{h}\in\wedge^2\hat{W}_{sep}$ as the image of $g\wedge h\in \wedge^2\Z[G]$.
Let $\bar{b}\in \hat{W}_{sep}$ and $b$ be an arbitrary lifting of $\bar{b}$.
We define $\varphi(\bar{b})=b\wedge N_G(1)$, and it is easy to see that the morphism $\varphi$ does not depend on the choice of the lifting. 

Assume that $G=\{g_1,\cdots,g_n\}$, therefore 
$\wedge^2\hat{W}_{sep}$ has a $\mathbb Z$-basis $\{\bar{g}_i\wedge \bar{g}_j\}_{1\leq i<j\leq n-1}$, since $\bar{g}_n=-\sum_{i=1}^{n-1}\bar{g}_i$ in $\hat{W}_{sep}$.
Suppose $a=\sum_{1\leq i< j\leq n}a_{ij}g_i\wedge g_j\in \wedge^2\mathbb Z[G]$ but $\bar{a}=0\in \wedge^2\hat{W}_{sep}$. 
We obtain $\bar{a}=\sum_{1\leq i< j\leq n-1}(a_{ij}-a_{in}+a_{jn})\bar{g}_i\wedge \bar{g}_j$, since  $\bar{g}_n=-\sum_{i=1}^{n-1}\bar{g}_i$ in $\hat{W}_{sep}$, therefore $a_{ij}=a_{in}-a_{jn}$, and $a=\sum_{1\leq i< j\leq n}(a_{in}-a_{jn})g_i\wedge g_j$ in $\wedge ^2\mathbb Z[G]$. 
Let $b=\sum_{i=1}^{n}a_{in}g_i\in \Z[G]$, then it is clear that $\varphi(\bar{b})=(\sum_{i=1}^{n}a_{in}g_i)\wedge N_G(1)=a$ and this proves the exactness.

d. The forth square is induced by the exact sequences:
$
0\rightarrow\mathbb Z\stackrel{N}{\rightarrow} \mathbb Z[G]\rightarrow\hat{W}_{sep}\rightarrow 0
$.

Note that
$\delta$, $\delta'$, $\delta''$ and the right vertical arrow in the third square are isomorphisms when we take $p$-primary component.
By Lemma \ref{commutator} and the above commutative diagram,  one obtains $$(H^0(F,S^2(\hat{T}_{sep}))/\Dec)\{p\}\simeq  \mathrm{Coker}(H^2(G,\mathbb Z)\times H^2(G,\mathbb Z)\stackrel{\cup}{\longrightarrow}H^4(G,\mathbb Z))\{p\}.\qedhere$$
\end{proof}

Combining Theorem \ref{mainlem} with Proposition \ref{prop:cup}, we immediately have the following result.
\begin{cor}
	Let $K/F$ be a Galois extension with $\mathrm{Gal}(K/F)=G$ and $W=R_{K/F}^{(1)}(\mathbb G_{m,K})$. Suppose  $\#[G,G]$ is a power of $2$. Then for an odd prime $p$, 
	$$
	\bar{H}^3_{nr}(F(W),\mathbb Q/\mathbb Z(2))\{p\}\simeq H^3(G,K^*)\{p\}  \oplus\mathrm{Coker}(H^2(G,\mathbb Z)\times H^2(G,\mathbb Z)\stackrel{\cup}{\longrightarrow}H^4(G,\mathbb Z))\{p\}.
	$$
\end{cor}

If $G$ is abelian, then $\#[G,G]=1$, 
this satisfies the condition of the above Corollary, 
and we can give a more precise expression for this cokernel  by induction and Hochschild-Serre spectral sequence. 

If $G=G'\times G''$ and acts trivially on $\mathbb Z$, then the Hochschild-Serre spectral sequence
$$
E_2^{p,q} = H^p(G',H^q(G'',\mathbb Z))\implies H^{p+q}(G,\mathbb Z)
$$
degenerates, and there is the decomposition by the main theorem of \cite{Jannsen1990TheSO}
$$
H^{n}(G,\mathbb Z)\simeq \oplus_{p+q=n}H^p(G',H^q(G'',\mathbb Z)).
$$
Moreover, the following diagram is commutative:
\begin{equation}
\begin{tikzcd}
H^p(G',\mathbb Z)\arrow[d,"\simeq"]\arrow[phantom, r, description, "\times"]&H^q(G'',\mathbb Z)\arrow[d,"\simeq"]\arrow[r,"\cup"]&H^{p+ q}(G,\mathbb Z)\arrow[d,"projection"]\\
E_2^{p,0}\arrow[phantom, r, description, "\times"]&E_2^{0,q}\arrow[r,"\cup"]&E_2^{p,q}
\end{tikzcd}\label{cup}
\end{equation}

Since $G$ is finite and abelian, we can represent $G$ as the direct sum of cyclic groups: $G=\oplus_{i=1}^mC_i$ such that $\#C_{i-1}|\#C_i$ ($1<i\leq m$). We set $G_i=\oplus_{j=1}^{i}C_j$, $G=G_m$, $G_{i+1}=G_i\oplus C_{i+1}$.

\begin{lem}
Notations as above, let $d_{i}=(m-i)(m-i-1)/2$, then
$$\mathrm{Coker}(H^2(G,\mathbb Z)\times H^2(G,\mathbb Z)\stackrel{\cup}{\longrightarrow}H^4(G,\mathbb Z))\simeq \oplus_{i=1}^{m-2}G_i^{m-i-1} \simeq \oplus_{i=1}^{m-2}C_i^{d_{i}}.
$$ 
\end{lem}

\begin{proof}
Firstly, $H^2(G,\mathbb Z)\simeq E_2^{2,0}\oplus E_2^{0,2}$ by the Hochschild-Serre spectral sequence with $G=G_{m-1}\times C_m$ and $E_2^{r,s}=0$ when $s$ is even. 
Similarly, $H^4(G,\mathbb Z)\simeq E_2^{4,0}\oplus E_2^{2,2}\oplus E_2^{0,4}$. Therefore, from \eqref{cup}, we can divide into three steps: 

a) $\ \mathrm{Coker}(E_2^{0,2}\times E_2^{0,2}\stackrel{\cup}{\longrightarrow}E_2^{0,4})$: 
 This is trivial by \cite[Proposition 1.6.12]{neukirch2013} since $C_m$ is cyclic .

b) $ \mathrm{Coker}(E_2^{2,0}\times E_2^{0,2}\stackrel{\cup}{\longrightarrow}E_2^{2,2})$: 
Assume that $C_m=\mathbb Z/n_m\mathbb Z$. Since $E_2^{0,2}\simeq C_m=\mathbb Z/n_m\mathbb Z$, the cokernel of this cup product is equal to the cokernel of the map $H^2(G_{m-1},\mathbb Z)\stackrel{}{\longrightarrow}H^2(G_{m-1},\mathbb Z/n_m\mathbb Z)$ induced by $\mathbb Z\to \mathbb Z/n_m\mathbb Z$. 
The exact sequence 
$$
0\longrightarrow \Z  \stackrel{\times n_m}{\longrightarrow}  \Z\longrightarrow \Z/n_m\longrightarrow 0
$$  derives the exact sequence
$$
H^2(G_{m-1},\mathbb Z)\longrightarrow H^2(G_{m-1},\mathbb Z/n_m\mathbb Z)\longrightarrow H^3(G_{m-1},\mathbb Z)\stackrel{\times n_m}{\longrightarrow} H^3(G_{m-1},\mathbb Z).
$$
The last homomorphism is 0 because of the decomposition of $G$. 
Thus, by an easy induction, one has $ \mathrm{Coker}(E_2^{2,0}\times E_2^{0,2}\stackrel{\cup}{\longrightarrow}E_2^{2,2})\simeq H^3(G_{m-1},\mathbb Z)\simeq \oplus_{i=1}^{m-2}G_i.$

c) $ \mathrm{Coker}(E_2^{2,0}\times E_2^{2,0}\stackrel{\cup}{\longrightarrow}E_2^{4,0})$: 
This can be obtained by an induction from the steps a) and b).

Finally, we obtain the result by induction:
\begin{align*}
\mathrm{Coker}(H^2(G,\mathbb Z)&\times H^2(G,\mathbb Z)\stackrel{\cup}{\longrightarrow}H^4(G,\mathbb Z))
\simeq \oplus_{i=1}^{m-1}H^3(G_{i},\mathbb Z)\\
&\simeq \oplus_{i=1}^{m-2}G_i^{m-i-1} \simeq \oplus_{i=1}^{m-2}C_i^{d_{i}}. \qedhere
\end{align*}
\end{proof}

 We immediately obtain the following theorem by the previous lemmas.
\begin{thm}
	 Let $W=R_{K/F}^{(1)}(\mathbb G_{m,K})$ and $K/F$  an abelian extension  with $\mathrm{Gal}(K/F)= G\simeq\oplus_{i=1}^mC_i$ such that each $C_i$ is cyclic and $\#C_{i-1}|\#C_i$ $(1<i\leq m)$. Denote $d_{i}=(m-i)(m-i-1)/2$. Then for any odd prime $p$, 
	$$
		\bar{H}^3_{nr}(F(W),\mathbb Q/\mathbb Z(2))\{p\}\simeq 	H^3(G,K^*)\{p\}  \oplus\oplus_{i=1}^{m-2}C_i\{p\}^{d_{i}}.
	$$
	Moreover, if $K/F$ has odd degree, then
	$$
	\bar{H}^3_{nr}(F(W),\mathbb Q/\mathbb Z(2))\simeq H^3(G,K^*)  \oplus \oplus_{i=1}^{m-2}C_i^{d_{i}}.
	$$
\end{thm}
\begin{proof}
	By (\ref{abelian}), if $\Gal(K/F)$ has odd order, $2$-primary part of $\bar{H}^3_{nr}(F(W),\mathbb Q/\mathbb Z(2))$ is trivial, the proof then follows.
\end{proof}

In particular, if $m=2$, we deduce:

\begin{cor}
Assume $W=R_{K/F}^{(1)}(\mathbb G_{m,K})$  and $K/F$  an abelian extension  with  $\mathrm{Gal}(K/F)=G\simeq C_1\oplus C_2$ such that $C_1$ and $C_2$ are cyclic groups, then $$\bar{H}^3_{nr}(F(W),\mathbb Q/\mathbb Z(2))\{p\}= H^3(G,K^*)\{p\}$$ for any odd prime $p$. In particular, if $G$ has odd order, then $$\bar{H}^3_{nr}(F(W),\mathbb Q/\mathbb Z(2))= H^3(G,K^*).$$
\end{cor}

\begin{rem*}\label{rem: K}
	If $K/F$ is a Galois extension of local fields, then $H^3(G,K^*)=0$ by \cite[Corollary 7.2.2]{neukirch2013}.
	If $K/F$ is a Galois extension of global fields, we consider the exact sequence
	$$
	1\longrightarrow K^*\longrightarrow I_K\longrightarrow C_K\longrightarrow 1
	$$
	where $I_K$ is the idèle group and $C_K$ is the idèle class group. 
	This induces the following commutative diagram with the first row exact \cite[Chapter VIII]{neukirch2013}
	\[
	\begin{tikzcd}
		H^2(G,I_K)\arrow[r,"\psi"]\arrow[d,"\simeq"]&H^2(G,C_K)\arrow[r]\arrow[dd,"\simeq","inv_{K|F}"']&H^3(G,K^*)\arrow[r]& H^3(G,I_K)\arrow[d, equal]\\
		\oplus_{\mathfrak p}H^2(G_{\mathfrak p},K^*_{\mathfrak p})\arrow[d,"\simeq","inv_{\mathfrak p}"']&&& \oplus_{\mathfrak p}H^3(G_{\mathfrak p},K_{\mathfrak p}^*)\arrow[d, equal]\\
		\oplus_{\mathfrak p}{1\over [K_{\mathfrak p}:F_{\mathfrak p}] }\mathbb Z/\mathbb Z\arrow[r]&{1\over [K:F] }\mathbb Z/\mathbb Z&& 0
	\end{tikzcd}
	\]
	where the sum run all prime $\mathfrak p$ over $F$ and $G_{\mathfrak p}$ is the decomposition subgroup of $G$  with some prime $\mathfrak B$ above  $\mathfrak p$. 
	Therefore $H^3(G,K^*)\simeq \mathrm{coker}\ \psi$.
	In general $\psi$ is  not surjective, such as $F=\Q$ and $K=\Q(\sqrt{-1},\sqrt{17})$.
\end{rem*}

\begin{exa*} \label{exa:un}
 Let $F=\mathbb{Q}(\sqrt{-3})$, $K/F$ an abelian field and  $W=R_{K/F}^{(1)}(\mathbb G_{m,K})$.

(1) Let $K=F(\sqrt[3]{2}, \sqrt[3]{7})$, 
then $\mathrm{Gal}(K/F)=G\simeq\mathbb{ Z}/3 \oplus \mathbb{ Z}/3$. By Remark \ref{rem: K}, one obtains $H^3(G,K^*)=0$ since $K_{\frak \nu}/F_{\frak p}$ has degree $9$, where $\nu\mid \frak p$ are places over $7$. So $$\bar{H}^3_{nr}(F(W),\mathbb Q/\mathbb Z(2)) = 0.$$

(2) Let $K=F(\sqrt[3]{2}, \sqrt[3]{3})$, 
then $\mathrm{Gal}(K/F)=G\simeq\mathbb{ Z}/3\oplus \mathbb{ Z}/3$. Its local Galois groups are cyclic, which implies  $H^3(G,K^*)=\mathbb{ Z}/3$ by Remark \ref{rem: K}. So $$\bar{H}^3_{nr}(F(W),\mathbb Q/\mathbb Z(2))= \Z/3.$$

(3) Let $K=F(\sqrt[3]{2},\sqrt[3]{3}, \sqrt[3]{7})$, 
then $\mathrm{Gal}(K/F)=G\simeq(\mathbb{ Z}/3)^3$. Its local Galois groups are $0, \mathbb{ Z}/3 \text{ or }(\mathbb{ Z}/3)^2$, which implies  $H^3(G,K^*)=\mathbb{ Z}/3$ by Remark \ref{rem: K}. So $$\bar{H}^3_{nr}(F(W),\mathbb Q/\mathbb Z(2))= (\Z/3)^2.$$
\end{exa*}

\bibliography{ref}
\nocite{serre}

\end{document}